\documentclass[number]{amsart}

\usepackage{amssymb,amsmath,amsfonts,stmaryrd,mathrsfs,amscd}
\usepackage[sans]{dsfont}
\usepackage{fouridx,chemarrow}
\usepackage[dvips,cmtip,all]{xy}
\usepackage{setspace,paralist}
\usepackage[linkcolor=black]{hyperref}
\usepackage{lineno,color,doi}
\usepackage{tikz,mathabx} 
\usetikzlibrary{matrix,arrows,decorations.pathmorphing}

\newtheorem{theorem}[]{Theorem}

\theoremstyle{definition}
\newtheorem{definition}[]{Definition}
\newtheorem{lemma}[]{Lemma}

\newtheorem{proposition}[]{Proposition}
\newtheorem{question}[]{Question}

\newtheorem{example}[]{Example}

\newtheorem*{nontheorem}{Theorem}{\bf}{\it}

\numberwithin{equation}{section}

\theoremstyle{remark}
\newtheorem{remark}[definition]{Remark}

\numberwithin{equation}{section}

\def\length{\operatorname{length}}

\begin{document}

\title{On additivity of local entropy under flat extensions}

\author{Mahdi Majidi-Zolbanin}
\address{Department of Mathematics, LaGuardia Community College of the City University of New York, 31-10 Thomson Avenue, Long Island City, NY 11101}
\email{mmajidi-zolbanin@lagcc.cuny.edu}

\thanks{The author received funding from $\mathrm{C}^3\mathrm{IRG}$ (round 10) grant provided by the City University of New York.}

%    General info

\date{September 7, 2014}

%\dedicatory{}


\begin{abstract}
Let $f\colon(R,\mathfrak{m})\rightarrow S$ be a local homomorphism of Noetherian local rings. Consider two endomorphisms \textit{of finite length} (i.e., with zero-dimensional closed fibers) $\varphi\colon R\rightarrow R$ and $\psi\colon S\rightarrow S$, satisfying $\psi\circ f=f\circ\varphi$. Then $\psi$ induces a finite length endomorphism $\overline{\psi}\colon S/f(\mathfrak{m})S\rightarrow S/f(\mathfrak{m})S$. When $f$ is flat, under the assumption that $S$ is Cohen-Macaulay we prove an additivity formula: $h_{\mathrm{loc}}(\psi)=h_{\mathrm{loc}}(\varphi)+h_{\mathrm{loc}}(\overline{\psi})$ for \textit{local entropy}. 
\end{abstract}

\maketitle

\section{Introduction}
All rings in this note are assumed to be Noetherian, local, commutative and with identity element $1$. 

The notion of \textit{local entropy} associated with an endomorphism \textit{of finite length} of a Noetherian local ring was introduced in~\cite{MajMiaSzp}. We recall a few definitions and results from~\cite{MajMiaSzp}.
\begin{definition}[\textnormal{\cite[Definition~1]{MajMiaSzp}}]
\label{Defin:1}
A local homomorphism \(f:(R,\mathfrak{m})\rightarrow(S,\mathfrak{n})\) of Noetherian local rings is said to be \textit{of finite length}, if one of the following equivalent conditions holds:
\begin{compactenum}
\item[a)] \(f(\mathfrak{m})S\) is \(\mathfrak{n}\)-primary;
\item[b)] The closed fiber of $f$ has dimension zero;
\item[c)] If \(\mathfrak{p}\) is a prime ideal of \(S\) such that \(f^{-1}(\mathfrak{p})=\mathfrak{m}\), then \(\mathfrak{p}=\mathfrak{n}\);
\item[d)] If \(\mathfrak{q}\) is any \(\mathfrak{m}\)-primary ideal of \(R\), then \(f(\mathfrak{q})S\) is \(\mathfrak{n}\)-primary.
\end{compactenum}
\end{definition}
\begin{definition}[\textnormal{\cite[Definition~5]{MajMiaSzp}}]
\label{Defin:2}
A \textit{local algebraic dynamical system} $(R,\varphi)$ consists of a Noetherian local ring $R$ and an endomorphism of finite length $\varphi\colon R\rightarrow R$. By a \textit{morphism} $f\colon(R,\varphi)\rightarrow(S,\psi)$ between two local algebraic dynamical systems we mean a local homomorphism $f\colon R\rightarrow S$ satisfying  $\psi\circ f=f\circ\varphi$. 
\end{definition}
\begin{nontheorem}[{\cite[Theorem~1]{MajMiaSzp}}]
Let \((R,\mathfrak{m},\varphi)\) be a local algebraic dynamical system. Write $\varphi^n$ for the $n$-fold composition of $\varphi$ with itself and let $\length_R(-)$ denote the length of an $R$-module of finite length. Then the limit
$$h_{\mathrm{loc}}(\varphi):=\lim_{n\rightarrow\infty}\frac{1}{n}\log\left(\length_R(R/\varphi^n(\mathfrak{m})R)\right)$$
exists and is a nonnegative real number.
\end{nontheorem}
The invariant $h_{\mathrm{loc}}(\varphi)$ is called the \textit{local entropy} of $\varphi$. Local entropy can be calculated using any $\mathfrak{m}$-primary ideal:
\begin{lemma}
\label{Lemma:1}
Let $(R,\mathfrak{m},\varphi)$ be a local algebraic dynamical system. If $\mathfrak{q}$ is an $\mathfrak{m}$-primary ideal of $R$, then
$$h_{\mathrm{loc}}(\varphi)=\lim_{n\rightarrow\infty}\frac{1}{n}\log\left(\length_R(R/\varphi^n(\mathfrak{q})R)\right).$$
\end{lemma}
\begin{proof}
This is a particular case of~\cite[Proposition~18]{MajMiaSzp}. We refer the reader to \textit{loc.~cit.} for a proof.
\end{proof}
This paper is concerned with the following question asked by Craig~Huneke:
\begin{question}
\label{Quest:1}
Let \(f\colon(R,\mathfrak{m},\varphi)\rightarrow(S,\mathfrak{n},\psi)\) be a morphism between two local algebraic dynamical systems. Then by definition of morphism, $\psi\circ f=f\circ\varphi$. The ideal $f(\mathfrak{m})S$ is quickly seen to be $\psi$-stable (i.e., $\psi\left(f(\mathfrak{m})S\right)\subseteq f(\mathfrak{m})S$). Thus, $\psi$ induces a finite length endomorphism $\overline{\psi}\colon S/f(\mathfrak{m})S\rightarrow S/f(\mathfrak{m})S$. If $f$ is flat, does it hold that
\begin{equation}
\label{Equat:one}
h_{\mathrm{loc}}(\psi)=h_{\mathrm{loc}}(\varphi)+h_{\mathrm{loc}}(\overline{\psi})?
\end{equation}
\end{question}
We mention two cases in which the answer to Question~\ref{Quest:1} is affirmative: (1) When $\dim R=\dim S$, the question has an affirmative answer, as shown in~\cite[Proposition~20]{MajMiaSzp}. (2) By~\cite[Theorem~1]{MajMiaSzp} the local entropy of the Frobenius endomorphism of a local ring of positive prime characteristic $p$ and of dimension $d$ is equal to $d\cdot\log p$. Hence, when $R$ and $S$ are of positive prime characteristic $p$, and $\varphi$ and $\psi$ are their Frobenius endomorphisms, respectively, then Equation~\ref{Equat:one} reduces to 
$$(\dim S)\cdot\log p=(\dim R)\cdot\log p + (\dim S/f(\mathfrak{m})S)\cdot\log p,$$ which holds, since $f$ is flat (see, e.g.,~\cite[Theorem~15.1]{Matsumura2}). The aim of this paper is to give an affirmative answer to Question~\ref{Quest:1} in the particular case when $S$ is Cohen-Macaulay. The question remains open in the general (non Cohen-Macaulay) case.
\section{Main result}
We will use the following Flatness Criterion in the proof of our main result, Theorem~\ref{Theorem:1}, as well as in Example~\ref{Example:1}. The reader can find a proof of this result in~\cite[Corollary to Theorem~22.5]{Matsumura2}.
\begin{nontheorem}[Flatness Criterion]
\textit{Let \(f\colon(R,\mathfrak{m})\rightarrow(S,\mathfrak{n})\) be a local homomorphism of Noetherian local rings and let $M$ be a finite $S$-module. For $y_1,\ldots,y_n\in\mathfrak{n}$ write $\overline{y}_i$ for the images of $y_i$ in $S/f(\mathfrak{m})S$. Then the following conditions are equivalent:}
\begin{compactenum}
\item[a)] \textit{$y_1,\ldots,y_n$ is an $M$-regular sequence and $M/\sum_1^ny_iM$ is flat over $R$;}
\item[b)] \textit{$\overline{y}_1,\ldots,\overline{y}_n$ is an $(M/f(\mathfrak{m})M)$-regular sequence and $M$ is flat over $R$.}
\end{compactenum}
\end{nontheorem}
We will also need the following elementary statement:
\begin{proposition}
\label{Propos:5}
Let \(f:(R,\mathfrak{m})\rightarrow S\) be a local homomorphism of finite length of Noetherian local rings. Let \(M\) be an \(R\)-module of finite length. Then
\begin{compactenum}
\item[a)] \(M\otimes_RS\) is of finite length as an \(S\)-module.
\item[b)] \(\length_S(M\otimes_R S)\leq\length_R(M)\cdot\length_S(S/f(\mathfrak{m})S)\).
\item[c)] If \(f\) is \textit{flat}, then \(\length_S(M\otimes_R S)=\length_R(M)\cdot\length_S(S/f(\mathfrak{m})S)\).
\end{compactenum}
\end{proposition}
\begin{proof}
By induction on $\length_R(M)$.
\end{proof}
We begin with showing that without assuming flatness, an inequality will still hold between local entropies:
\begin{proposition}
Let \(f\colon(R,\mathfrak{m},\varphi)\rightarrow(S,\mathfrak{n},\psi)\) be a morphism between two local algebraic dynamical systems and let $\overline{\psi}\colon S/f(\mathfrak{m})S\rightarrow S/f(\mathfrak{m})S$ be the endomorphism induced by $\psi$. Then the following inequality holds:
$$h_{\mathrm{loc}}(\psi)\leq h_{\mathrm{loc}}(\varphi)+h_{\mathrm{loc}}(\overline{\psi}).$$
\end{proposition}
\begin{proof}
The composition of maps $R\stackrel{f}{\rightarrow}S\rightarrow S/\psi^n(\mathfrak{n})S$ gives a local homomorphism of finite length $R\rightarrow S/\psi^n(\mathfrak{n})S$. Applying Proposition~\ref{Propos:5}, we can write:
\begin{eqnarray*}
\length_S(S/\psi^n(\mathfrak{n})S)&=&\length_S\left((R/\varphi^n(\mathfrak{m})R)\otimes_R (S/\psi^n(\mathfrak{n})S\right)\\ &\leq&\length_R(R/\varphi^n(\mathfrak{m})R)\cdot\length_S(S/(f(\mathfrak{m})S+\psi^n(\mathfrak{n})S)).
\end{eqnarray*}
We obtain the desired inequality by applying logarithm, dividing by $n$ and taking limits as $n\rightarrow\infty$.
\end{proof}
\begin{remark}
\label{Remark:1}
In a Cohen-Macaulay Noetherian local ring of dimension $d$, a sequence of $d$ elements in the maximal ideal form a system of parameters if and only if they form a (maximal) regular sequence. We will use this fact frequently in our proof of Theorem~\ref{Theorem:1}. The reader can find a proof of this fact in~\cite[Theorem~17.4]{Matsumura2}.
\end{remark}
We now give an affirmative answer to Question~\ref{Quest:1} in the particular case when $S$ is Cohen-Macaulay:
\begin{theorem}
\label{Theorem:1}
Let \(f\colon(R,\mathfrak{m},\varphi)\rightarrow(S,\mathfrak{n},\psi)\) be a flat morphism between two local algebraic dynamical systems and let $\overline{\psi}\colon S/f(\mathfrak{m})S\rightarrow S/f(\mathfrak{m})S$ be the endomorphism induced by $\psi$. If $S$ is Cohen-Macaulay, then
\begin{equation}
\label{Equat:0}
h_{\mathrm{loc}}(\psi)=h_{\mathrm{loc}}(\varphi)+h_{\mathrm{loc}}(\overline{\psi}).
\end{equation}
\end{theorem}
\begin{proof}
Since $f$ is flat, the Cohen-Macaulayness of $S$ implies that the rings $R$ and $S/f(\mathfrak{m})S$ are also Cohen-Macaulay (see, e.g.,~\cite[Corollary to Theorem~23.3]{Matsumura2}). Since $S/f(\mathfrak{m})S$ is Cohen-Macaulay, there exists a (non-unique) sequence of elements $y_1,\ldots,y_{d^\prime}\in\mathfrak{n}$ of length $d^\prime=\dim (S/f(\mathfrak{m})S)$, whose images in $S/f(\mathfrak{m})S$ form an $(S/f(\mathfrak{m})S)$-regular sequence. Note that by the Flatness Criterion mentioned earlier, $y_1,\ldots,y_{d^\prime}$ is an $S$-regular sequence. Let $\mathfrak{q}^\prime\subset S$ be the ideal generated by $y_1,\ldots,y_{d^\prime}$. We claim that for any integer $n\geq0$, the ring $S/\psi^n(\mathfrak{q}^\prime)S$ is flat over $R$ via the composition of maps
\begin{equation}
\label{Equat:1}
R\stackrel{f}{\rightarrow}S\rightarrow S/\psi^n(\mathfrak{q}^\prime)S.
\end{equation}
Since $R\stackrel{f}{\rightarrow}S$ is flat, the claim will be established by the Flatness Criterion, if we can show that the images of $\psi^n(y_1),\ldots,\psi^n(y_{d^\prime})$ in $S/f(\mathfrak{m})S$ form an $(S/f(\mathfrak{m})S)$-regular sequence. These images coincide with elements $$\overline{\psi}^n(\overline{y}_1),\ldots,\overline{\psi}^n(\overline{y}_{d^\prime}),$$ where $\overline{y}_i$ is the image of $y_i$ in $S/f(\mathfrak{m})S$. That $\overline{\psi}^n(\overline{y}_1),\ldots,\overline{\psi}^n(\overline{y}_{d^\prime})$ is an $(S/f(\mathfrak{m})S)$-regular sequence is an immediate consequence of Remark~\ref{Remark:1}, the fact that $\overline{y}_1,\ldots,\overline{y}_{d^\prime}$ is a maximal $(S/f(\mathfrak{m})S)$-regular sequence, and the fact that $\overline{\psi}^n$ is an endomorphism of finite length of $S/f(\mathfrak{m})S$ (hence, the image under $\overline{\psi}^n$ of any system of parameters is again a system of parameters in $S/f(\mathfrak{m})S$).\par
Now let  $x_1,\ldots,x_d\in\mathfrak{m}$ be an $R$-regular sequence of length $d=\dim R$ and let $\mathfrak{q}\subset R$ be the ideal generated by $x_1,\ldots,x_d$. By Remark~\ref{Remark:1}, $\mathfrak{q}$ is generated by a system of parameters in $R$. By the flatness of $S/\mathfrak{q}^\prime S$ over $R$ via the composition of maps shown in~\ref{Equat:1} (taking $n=0$), the images of $f(x_1),\ldots,f(x_d)$ in $S/\mathfrak{q}^\prime S$ form an $(S/\mathfrak{q}^\prime S)$-regular sequence. This means $y_1,\ldots,y_{d^\prime},f(x_1),\ldots,f(x_d)$ is an $S$-regular sequence. Moreover, since $f$ is flat, $$d+d^\prime=\dim R+\dim (S/f(\mathfrak{m})S)=\dim S$$ (see, e.g.,~\cite[Theorem~15.1]{Matsumura2}). Hence, $\{y_1,\ldots,y_{d^\prime},f(x_1),\ldots,f(x_d)\}$ is a system of parameters in $S$, by Remark~\ref{Remark:1}. Let $\mathfrak{Q}\subset S$ be the ideal generated by $$y_1,\ldots,y_{d^\prime},f(x_1),\ldots,f(x_d).$$ We note that for any integer $n\geq0$
\begin{equation}
\label{Equat:2}
\frac{R}{\varphi^n(\mathfrak{q})R}\otimes_R\frac{S}{\psi^n(\mathfrak{q}^\prime)S}\cong\frac{S}{f(\varphi^n(\mathfrak{q}))S+\psi^n(\mathfrak{q}^\prime)S}\cong\frac{S}{\psi^n(\mathfrak{Q})S},
\end{equation}
where the last isomorphism quickly follows from the fact that $\psi\circ f=f\circ\varphi$. Since $S/\psi^n(\mathfrak{q}^\prime)S$ is flat over $R$ and 
$$\dim (S/\psi^n(\mathfrak{q}^\prime)S)=\dim S-d^\prime=\dim S-\dim (S/f(\mathfrak{m})S)=\dim R,$$
the homomorphism $R\rightarrow S/\psi^n(\mathfrak{q}^\prime)S$ obtained by composing the maps given in~\ref{Equat:1} is in fact, of finite length. Hence, Proposition~\ref{Propos:5}-c) applies and from~\ref{Equat:2} we obtain
\begin{eqnarray*}
\length_S\left(S/\psi^n(\mathfrak{Q})S\right)&=&\length_S\big(\frac{R}{\varphi^n(\mathfrak{q})R}\otimes_R\frac{S}{\psi^n(\mathfrak{q}^\prime)S}\big)\\
&=&\length_R\left(R/\varphi^n(\mathfrak{q})R\right)\cdot\length_S\left(S/[f(\mathfrak{m})S+\psi^n(\mathfrak{q}^\prime)S]\right).
\end{eqnarray*}
After applying logarithm to both sides, dividing by $n$ and taking limits as $n\rightarrow\infty$, we obtain the desired Equation~\ref{Equat:0}, by Lemma~\ref{Lemma:1}.
\end{proof}
\begin{example}
\label{Example:1}
In this example we will apply Theorem~\ref{Theorem:1} to calculate local entropy of a specific endomorphism. Consider the endomorphism of the ring $(\mathbb{Z}/2\mathbb{Z})\llbracket x,y,w,s\rrbracket$ that maps $x, y, w $ and $s$ to $x^3+s^3, y^3, w^5+x^2$ and $xs^2$, respectively. This endomorphism is of finite length, because if $\mathfrak{p}$ is a minimal prime ideal of $(x^3+s^3, y^3, w^5+x^2,xs^2)$, then as one can quickly see, $\mathfrak{p}=(x,y,w,s)$. One can also verify quickly that the ideal $(s^6,y^3+x^2)$ is stable under this endomorphism. Thus, we obtain an induced ring endomorphism:
$$\psi\colon\frac{(\mathbb{Z}/2\mathbb{Z})\llbracket x,y,w,s\rrbracket}{(s^6,y^3+x^2)}\rightarrow\frac{(\mathbb{Z}/2\mathbb{Z})\llbracket x,y,w,s\rrbracket}{(s^6,y^3+x^2)}.$$ 
To abbreviate notation we will write $S$ for the ring $(\mathbb{Z}/2\mathbb{Z})\llbracket x,y,w,s\rrbracket/(s^6,y^3+x^2)$. Our goal in this example is to calculate $h_{\mathrm{loc}}(\psi)$, the local entropy of $\psi$. We will do this by constructing a flat homomorphism into the ring $S$ and then using Theorem~\ref{Theorem:1}. Note that $S$ is Cohen-Macaulay by virtue of being a complete intersection. Let $R=(\mathbb{Z}/2\mathbb{Z})\llbracket y\rrbracket$ and let $\varphi\colon R\rightarrow R$ be the endomorphism that maps $y$ to $y^3$. To define a homomorphism $f\colon R\rightarrow S$ set $f(y)=y$ and then extend it linearly to all of $R$. It is evident that $f\circ\varphi=\psi\circ f$. From the Flatness Criterion that was stated earlier, it quickly follows that $f$ is flat. Hence, by Theorem~\ref{Theorem:1}
\begin{eqnarray*}
h_{\mathrm{loc}}(\psi)&=&h_{\mathrm{loc}}(\varphi)+h_{\mathrm{loc}}(\overline{\psi})\\
&=&\log3+h_{\mathrm{loc}}(\overline{\psi}),
\end{eqnarray*}
where as usual $\overline{\psi}$ is the endomorphism induced by $\psi$ on $S/yS$. (That $h_{\mathrm{loc}}(\varphi)=\log3$ can be calculated quickly, using the definition of local entropy.) The ring $S/yS$ is isomorphic to $S^\prime:=(\mathbb{Z}/2\mathbb{Z})\llbracket x,w,s\rrbracket/(s^6,x^2)$ and $\overline{\psi}$ maps $x, w$ and $s$ to $s^3, w^5$ and $xs^2$, respectively.  To calculate $h_{\mathrm{loc}}(\overline{\psi})$, we construct another flat homomorphism, this time into $S^\prime$. Let $R^\prime:=(\mathbb{Z}/2\mathbb{Z})\llbracket w\rrbracket$ and let $\varphi^\prime\colon R^\prime\rightarrow R^\prime$ be the endomorphism that maps $w$ to $w^5$. To define a homomorphism $f^\prime\colon R^\prime\rightarrow S^\prime$ set $f^\prime(w)=w$ and then extend it linearly to all of $R^\prime$. Again it is evident that $f\circ\varphi=\psi\circ f$ and the flatness of $f^\prime$ quickly follows from the Flatness Criterion that was stated earlier. By Theorem~\ref{Theorem:1}, and using the fact that the local entropy of an endomorphism of a zero-dimensional local ring is zero (\cite[Corollary~16]{MajMiaSzp}), we obtain $h_{\mathrm{loc}}(\overline{\psi})=\log5$. Hence, $h_{\mathrm{loc}}(\psi)=\log3+\log5$.
\end{example}


\begin{thebibliography}{99}
\bibitem{MajMiaSzp} M.~Majidi-Zolbanin, N.~Miasnikov, L.~Szpiro: Entropy and flatness in local algebraic dynamics, \textit{Publicacions Matem\`atiques} \textbf{57(2)} (2013), 509-544. \doi{10.5565/PUBLMAT_57213_12}
\bibitem{Matsumura2} H.~Matsumura: ``\textit{Commutative Ring Theory}'', Cambridge Studies in Advanced Mathematics \textbf{8}, Cambridge University Press, Cambridge, 1986.
\end{thebibliography}
\end{document}